\documentclass[11pt]{article}

\usepackage{eurosym}
\usepackage{graphicx}
\usepackage{amsmath}

\usepackage{amsthm}

\newtheorem{thm}[equation]{Theorem}
\newtheorem{lem}[equation]{Lemma}
\newtheorem{cor}[equation]{Corollary}
\newtheorem{prop}[equation]{Proposition}

\newtheorem*{thm*}{Theorem}
\newtheorem*{prop*}{Proposition}
\newtheorem*{cor*}{Corollary}
\newtheorem*{lem*}{Lemma}
\newtheorem*{MT*}{Main Theorem}

\theoremstyle{definition} %

\newtheorem*{defn*}{Definition}

\theoremstyle{remark} %

\newtheorem*{rmk*}{Remark}
\newtheorem*{rmks*}{Remarks}

\newtheoremstyle{exercise}
  {3pt}
  {3pt}
  {\small}
  {}
  {\sc\small}
  {.}
  {.5em}
   {}     
  {}

\theoremstyle{exercise}

\makeatletter
{\renewcommand{\theequation}{#1}}%
{\renewcommand{\theequation}{\arabic{equation}}\addtocounter{equation}{-1}\global\@ignoretrue}
\makeatother

\makeatletter
{\renewcommand{\theequation}{#1}\begin{eqnarray}}%
{\end{eqnarray}\renewcommand{\theequation}{\arabic{equation}}\addtocounter{equation}{-1}\global\@ignoretrue}
\makeatother

\makeatletter
\newenvironment{borel}[1]%
{\smallskip \refstepcounter{equation}\noindent{\textbf{\theequation.} }{{\textbf{#1.}}}}%
{\smallskip \global\@ignoretrue}
\makeatother

\makeatletter
{\smallskip \refstepcounter{equation}\noindent{\textbf{\theequation.} }{{\textbf{#1}}}}%
{\smallskip \global\@ignoretrue}
\makeatother

\makeatletter
{\smallskip \refstepcounter{equation}{\sc \theequation}{\sc (#1).}}%
{\smallskip \global\@ignoretrue}
\makeatother

\makeatletter
{\smallskip\refstepcounter{equation}\noindent{\textbf{\theequation.}}{\textsl{ #1.}}}%
{\smallskip\global\@ignoretrue}
\makeatother

\makeatletter
\newenvironment{borel*}%
{\smallskip \refstepcounter{equation}\noindent{\textbf{\theequation.}}}%
{\global\@ignoretrue}
\makeatother

\newcommand{\flist}[1]{\hangindent\leftmargini\textup{(1)}\hskip\labelsep {#1}%
\begin{enumerate}%
\setcounter{enumi}{1}%
}

\theoremstyle{plain}

\newtheorem*{negthm}{Negative Theorem}

\renewcommand{\thetable}{\thesection\alph{table}}
\renewcommand{\thefigure}{\thesection}

\newcommand{\dx}{\mathrm{d}x}
\newcommand{\ddy}{\mathrm{d}^2y}
\newcommand{\dy}{\mathrm{d}y}

\renewcommand{\flist}[1]{\setcounter{enumi}{1}\hangindent\leftmargini\textup{\labelenumi}\hskip\labelsep {#1}%
\begin{enumerate}%
\setcounter{enumi}{1}%
}

\renewcommand{\labelenumi}{(\theenumi)}

\numberwithin{equation}{section}

\newcommand{\fix}{{\mathrm{fix}}}
\renewcommand{\pm}{{\mathrm{pari}}}

\newcommand{\eRoR}{(\text{eRoR})}

\renewcommand{\j}{J}
\newcommand{\n}{N}

\begin{document}

\title{Finding good bets in the lottery,\\ 
and why you shouldn't take them\footnote{Final version published in \emph{American Mathematical Monthly,} vol.~117 no.~1 (2010), pp.~3--26.}} 
\author{Aaron Abrams and Skip Garibaldi}

\date{}



\maketitle

Everybody knows that the lottery is a bad investment.
But do you know why?  How do you know?  

For most lotteries, the obvious answer is 
obviously correct:  lottery operators are 
running a business, and we can assume they have set up the game so that they 
make money.  If they make money, they must be paying out less than they are 
taking in; so on average, the ticket buyer loses money.  This reasoning applies,
for example, to the policy games formerly run by organized crime described in
\cite{numbers} and \cite{X}, and to the (essentially identical) Cash 3 and Cash 4 
games currently offered in the state of Georgia, where the authors reside.  This reasoning also 
applies to Las Vegas-style gambling.  (How do you think the Luxor can afford to 
keep its spotlight lit?)

However, the question becomes less trivial for games with \emph{rolling jackpots}, like Mega 
Millions (currently played in 12 of the 50 U.S.\  states), Powerball (played in 30 states), 
and various U.S.\  state lotteries.  In these games, if no player wins the largest prize 
(the \emph{jackpot}) in any particular drawing,
then that money is ``rolled over'' and increased for the next drawing.  On
average the operators of the game still make money, but for a particular drawing, 
one can imagine that a sufficiently large jackpot would give a lottery ticket a positive
average return (even though the probability of winning the jackpot with a single ticket 
remains extremely small).  Indeed, for any particular drawing, it is easy enough to 
calculate the expected rate of return using the formula in \eqref{actual.E}.  This has been done in the literature for lots of drawings (see, e.g.,~\cite{Matheson}),
and, sure enough, sometimes the expected rate of return is positive.  In this situation,
why is the lottery a bad investment?

Seeking an answer to this question, we began by studying historical lottery data.
In doing so, we were surprised both by which lotteries offered the good bets, and also by just how
good they can be.  We almost thought we should invest in the lottery!  So we 
were faced with several questions; for example, are there any rules of thumb to help pick out the 
drawings with good rates of return?  One jackpot winner said she only bought lottery tickets when the announced jackpot 
was at least \$100 million \cite{AP:minbet}.  Is this a good idea?  (Or perhaps a
modified version, replacing the threshold with something less arbitrary?)
Sometimes the announced jackpots of these games are truly enormous, 
such as on March 9, 2007, when Mega Millions announced a \$390 million prize.  
Would it have been a good idea to buy a ticket for that drawing?  And our real question,
in general, is the following:   on the occasions that the lottery offers a positive rate of return, 
is a lottery ticket ever a good investment?  And how can we tell? 
In this paper we document our findings.  Using elementary mathematics and economics,
we can give satisfactory answers to these questions.

We should come clean here and admit that to this point we have been deliberately
conflating several notions.  By a ``good bet" (for instance in the title of this paper)
we mean any wager with a positive rate of
return.  This is a mathematical quantity which is easily computed.  A ``good investment" is 
harder to define, and must take into account risk.  This is where things really get 
interesting, because as any undergraduate economics major knows, mathematics
alone does not provide the tools to determine when a good bet is a good investment
(although a bad bet is always a bad investment!).
To address this issue we therefore leave the domain of mathematics and enter a 
discussion of some basic economic theory, which, in Part \ref{risk} of the paper, succeeds
in answering 
our questions (hence the second part of the title).  And by the way, a ``good idea" is 
even less formal:  independently of your financial goals and strategies, you might 
enjoy playing the lottery for a variety of reasons.  We're not going to try to stop you.

To get started, we build a mathematical model of a lottery drawing.  Part \ref{setup} 
of this paper (\S\S\ref{MMPB}--\ref{model}) describes the model in detail:  it has three 
parameters ($f, F, t$) that depend only on the lottery and not on a particular drawing, and two 
parameters ($\n,\j$) that vary from drawing to drawing.  Here $\n$ is the total ticket sales and 
$\j$ is the size of the jackpot.  (The reader interested in a particular lottery can easily 
determine $f, F,$ and $t$.)  The benefit of the general model, of course, is that it allows us to 
prove theorems about general lotteries.  The parameters are free enough that the results
apply to Mega Millions, Powerball, and many other smaller lotteries.

In Part \ref{rateofreturn} (\S\S\ref{E.sec}--\ref{MMPB.sec}) we use elementary calculus 
to derive criteria for determining, without too much effort, whether a given
drawing is a good bet.  We show, roughly speaking, that drawings with ``small'' ticket
sales (relative to the jackpot; the measurement we use is $\n/\j$, which should be less
than $1/5$) offer positive rates
of return, once the jackpot exceeds a certain easily-computed threshold.  Lotto Texas
is an example of such a lottery.  On the other hand, drawings with ``large" ticket sales 
(again, this means $\n / \j$ is larger than a certain cutoff, which is slightly larger than 1)
will always have negative rates of return.  As it happens, Mega Millions and Powerball 
fall into this category; in particular, no drawing of either of these two lotteries has ever 
been a good bet, including the aforementioned \$390 million jackpot.  Moreover, based
on these considerations we argue in Section \ref{MMPB.sec} that Mega Millions and 
Powerball drawings are likely to always be bad bets in the future also.

With this information in hand, we focus on those drawings that have positive 
expected rates of return, i.e.,~the good bets, and we ask, from an economic point of view, 
whether they 
can ever present a good investment.  If you buy a ticket, of course, you will most likely 
lose your dollar; on 
the other hand, there is a small chance that you will win big.  Indeed, this is the nature
of investing (and gambling):  every interesting investment offers the potential of gain 
alongside the risk of loss.  
If you view the lottery as a game, like playing roulette, then you are probably playing for 
fun and you are both willing and expecting to lose your dollar.  But what if you really want 
to make money?  Can you do it with the lottery?  More generally, how do you compare 
investments whose expected rates of return and risks differ?

In Part \ref{risk} of the paper (\S\S\ref{port}--\ref{lottery.port}) we discuss basic portfolio theory, 
a branch of economics that gives a concrete, quantitative answer to exactly this question.
Portfolio theory is part of a standard undergraduate economics curriculum, but it is not so 
well known to mathematicians.   
Applying portfolio theory to the lottery, we find, as one might expect, that even
when the returns are favorable, the risk of a lottery ticket is so large that an optimal
investment portfolio will allocate a negligible fraction of its assets to lottery tickets.  
Our conclusion, then, is unsurprising; to quote the movie \emph{War Games}, 
``the only winning move is not to play."\footnote{How about a nice game of chess?}

You might respond: ``So what?  I already knew that buying a lottery ticket was a bad 
investment."  And maybe you did.  But we thought we knew it too, until we discovered
the fantastic expected rates of return offered by certain lottery drawings!  The point
we want to make here is that if you want to actually \emph{prove} that the lottery is a bad 
investment, the existence of good bets shows that mathematics is not enough.
It takes economics in cooperation with mathematics to ultimately validate our intuition. 

\subsection*{Further reading} The lotteries described here are all modern variations on a lottery invented in Genoa in the 1600s, allegedly to select new senators \cite{Bellhouse1}.  The Genoese-style lottery became very popular in Europe, leading to interest by mathematicians including Euler; see, e.g., \cite{Euler} or \cite{Bellhouse2}.  The papers \cite{MG:US} and \cite{Ziemba} survey modern U.S.\  lotteries from an economist's perspective.  The book \cite{DrZ} gives a treatment for a general audience.   The conclusion of Part III of the present paper---that even with a very good expected rate of return, lotteries are still too risky to make good investments---has of course been observed before by economists; see \cite{MZB}.  Whereas we compare the lottery to other investments via portfolio theory, the paper \cite{MZB} analyzes a lottery ticket as an investment in isolation, using the Kelly criterion (described, e.g., in \cite{RotandoThorp} or \cite{Thorp}) to decide whether or not to invest.  
They also conclude that one shouldn't invest in the lottery, but for a different reason than
ours:  they argue that investing in the lottery using a strategy based on the Kelly criterion,
even under favorable conditions, is likely to take millions of years to produce positive returns.  
The mathematics required for their analysis is more sophisticated than the undergraduate-level 
material used in ours.

\part{The setup:  Modeling a lottery}
\label{setup}
\section{Mega Millions and Powerball} \label{MMPB}
\renewcommand{\thetable}{\thesection}

The Mega Millions and Powerball lotteries are similar in that in both, a player purchasing a \$1 ticket selects 5 distinct ``main" numbers (from 1 to 56 in Mega Millions and 1 to 55 in Powerball) and 1 ``extra" number (from 1 to 46 in Mega Millions and 1 to 42 in Powerball).  This extra number is not related to the main numbers, so, e.g., the sequence
\[
\text{main} = 4, 8, 15, 16, 23  \quad \text{and} \quad \text{extra} = 15
\]
denotes a valid ticket in either lottery.  The number of possible distinct tickets is $\binom{56}{5} 46$ for Mega Millions and $\binom{55}{5} 42$ for Powerball.

At a predetermined time, the ``winning" numbers are drawn on live television and the player wins a prize (or not) based on how many numbers on his or her ticket match the winning numbers.  The prize payouts are listed in Table \ref{eg.pay}.  A ticket wins only the best prize for which it qualifies; e.g., a ticket that matches all six numbers only wins the jackpot and not any of the other prizes.  We call the non-jackpot prizes \emph{fixed}, because their value is fixed.  (In this paper, we treat a slightly simplified version of the Powerball game offered from August 28, 2005 through the end of 2008.  The actual game allowed the player the option of buying a \$2 ticket that has larger fixed prizes.  Also, in the event of a record-breaking jackpot, some of the fixed prizes are also increased by a variable amount.  We ignore both of these possibilities.  The rules for Mega Millions also vary slightly from state to state,\footnote{Most notably, in California all prizes are pari-mutuel.} and we take the simplest and most popular version here.)
\begin{table}[ht]
\begin{center}
\begin{tabular}{|c|rr|rr|} \hline
&\multicolumn{2}{c|}{Mega Millions}&\multicolumn{2}{c|}{Powerball} \\ \hline
Match&\multicolumn{1}{c}{Payout}&\parbox{0.7in}{\# of ways to make this match}&\multicolumn{1}{c}{Payout}&\parbox{0.7in}{\# of ways to make this match} \\[3ex] \hline
5/5 + extra&\multicolumn{1}{c}{jackpot}&1&\multicolumn{1}{c}{jackpot}&1 \\
5/5&\$187,500&45&\$150,000&41\\
4/5 + extra&\$7,500&255&\$7,500&250\\
4/5&\$150&11,475&\$100&10,250 \\
3/5 + extra&\$150&12,750&\$100&12,250 \\
2/5 + extra&\$10&208,250&\$7&196,000\\
3/5&\$7&573,750&\$7&502,520\\
1/5 + extra&\$3&1,249,500&\$4&1,151,500\\
0/5 + extra&\$2&2,349,060&\$3&2,118,760\\ \hline
\end{tabular}
\caption{Prizes for Mega Millions and Powerball.  A ticket costs \$1.  Payouts for the 5/5 and 4/5 + extra prizes have been reduced by 25\% to approximate taxes.} \label{eg.pay}
\end{center}
\end{table}

The payouts listed in our table for the two largest fixed prizes are slightly different from those listed on the lottery websites, in that we have deducted federal taxes.  Currently, gambling winnings over \$600  are subject to federal income tax, and winnings over \$5000 are withheld at a rate of $25\%$; see  \cite{W2G} or \cite{tax}.  Since income tax rates vary from gambler to gambler, we use $25\%$ as an estimate of the tax rate.\footnote{We guess that most people who win the lottery will pay at least 25\% in taxes.  For anyone who pays more, the estimates we give of the jackpot value $\j$ for any particular drawing should be decreased accordingly.  This kind of change strengthens our final conclusion---namely, that buying lottery tickets is a poor investment.}  For example, the winner of the largest non-jackpot prize for the Mega Millions lottery receives not the nominal \$250,000 prize, but rather $75\%$ of that amount,
as listed in Table \ref{eg.pay}.  Because state tax rates on gambling winnings vary from state to state and Mega Millions and Powerball are each played in states that do not tax state lottery winnings (e.g., New Jersey \cite[p.~19]{NJ1040} and New Hampshire respectively), we ignore state taxes for these lotteries. 

\section{Lotteries with other pari-mutuel prizes} 

In addition to Mega Millions or Powerball, some states offer their own lotteries with rolling jackpots.  Here we describe the Texas (``Lotto Texas") and New Jersey (``Pick 6") games.  In both, a ticket costs \$1 and consists of 6 numbers (1--49 for New Jersey and 1--54 for Texas).  For matching 3 of the 6 winning numbers, the player wins a fixed prize of \$3.

All tickets that match 4 of the 6 winning numbers split a pot of $.05\n$ (NJ) or $.033\n$ (TX), where $\n$ is the net amount of sales for that drawing.  (As tickets cost $\$1$, as a number, $\n$ is the same as the total number of tickets sold.)  The prize for matching 5 of the 6 winning number is similar; such tickets split a pot of $.055\n$ (NJ) or $.0223\n$ (TX); these prizes are typically around \$2000, so we deduct 25\% in taxes from them as in the previous section, resulting in $.0413\n$ for New Jersey and $.0167\n$ for Texas.  (Deducting this 25\% makes no difference to any of our conclusions, it only slightly changes a few numbers along the way.)
Finally, the tickets that match all 6 of the 6 winning numbers split the jackpot.

How did we find these rates?  For New Jersey, they are on the state lottery website.  Otherwise, you can approximate them from knowing---for a single past drawing---the prize won by each ticket that matched 4 or 5 of the winning numbers, the number of tickets sold that matched 4 or 5 of the winning numbers, and the total sales $\n$ for that drawing.  (In the case of Texas, these numbers can be found on the state lottery website, except for total sales, which is on the website of a third party \cite{TXsales}.)  The resulting estimates may not be precisely correct, because the lottery operators typically round the prize given to each ticket holder to the nearest dollar.

As a matter of convenience, we refer to the prize for matching 3 of 6 as \emph{fixed}, the prizes for matching 4 or 5 of 6 as \emph{pari-mutuel}, and the prize for matching 6 of 6 as the \emph{jackpot}.  (Strictly speaking, this is an abuse of language, because the jackpot is also pari-mutuel in the usual sense of the word.)

\section{The general model} \label{model}
\renewcommand{\thetable}{\thesection\textsc{\alph{table}}}
\setcounter{table}{0}

We want a mathematical model that includes Mega Millions, Powerball, and the New Jersey and Texas lotteries described in the preceding section.  We model an individual drawing of such a lottery, and write $\n$ for the total ticket sales, an amount of money.  Let $t$ be the number of distinct possible tickets, of which:
\begin{itemize}
\item $t^{\fix}_1, t^{\fix}_2, \ldots, t^{\fix}_c$ win fixed prizes of (positive) amounts $a_1, a_2, \ldots, a_c$ respectively,
\item $t^{\pm}_1, t^{\pm}_2, \ldots, t^{\pm}_d$ split pari-mutuel pots of (positive) size $r_1 \n, r_2 \n, \ldots, r_{d} \n$ respectively, and
\item 1 of the possible distinct tickets wins a share of the (positive) jackpot $\j$.  More precisely, if $w$ copies of this 1 ticket are sold, then each ticket-holder receives $\j/w$.
\end{itemize}

Note that the $a_i$, the $r_i$, and the various $t$'s depend on the setup of the lottery,
whereas $N$ and $J$ vary from drawing to drawing.  Also, we mention a few technical points.  The prizes $a_i$, the number $\n$, and the jackpot $\j$ are denominated in units of ``price of a ticket".  For all four of our example lotteries, the tickets cost \$1, so, for example, the amounts listed in Table \ref{eg.pay} are the $a_i$'s---one just drops the dollar sign.  Furthermore, the prizes are \emph{the actual amount the player receives}.  We assume that taxes have already been deducted from these amounts, at whatever rate such winnings would be taxed.  (In this way, we avoid having to include tax in our formulas.)  Jackpot winners typically have the option of receiving their winnings as a lump sum or as an annuity; see, e.g., \cite{dope} for an explanation of the differences.  We take $\j$ to be the after-tax value of the lump sum, or---what is the essentially the same---the present value (after tax) of the annuity.  Note that this $\j$ is far smaller than the jackpot amounts announced by lottery operators, which is usually a total of the pre-tax annuity payments.  Some comparisons are shown in Table \ref{eg.jacks}.
\begin{table}[htb]
\begin{tabular}{|r@{/}c@{/}lc|rrr|}\hline
\multicolumn{3}{|c}{Date}&Game&\parbox{.75in}{Annuity jackpot (pre-tax)}&\parbox{.75in}{Lump sum jackpot (pre-tax)}&$\j$ (estimated) \\ \hline
4&07&2007&Lotto Texas&75m&45m&33.8m \\
3&06&2007&Mega Millions&390m&233m&175m \\ 
2&18&2006&Powerball&365m&177.3m&133m \\
10&19&2005&Powerball&340m&164.4m&123.3m \\ \hline
\end{tabular}
\caption{Comparison of annuity and lump sum jackpot amounts for some lottery drawings.  The value of $J$ is the lump sum minus tax, which we assume to be 25\%.  The letter `m' denotes millions of dollars.} \label{eg.jacks}
\end{table}

We assume that the player knows $\j$.  After all, the pre-tax value of the annuitized jackpot is announced publicly in advance of the drawing, and from it one can estimate $\j$.  For Mega Millions and Powerball, the lottery websites also list the pre-tax value of the cash jackpot, so the player only needs to consider taxes.

\subsection*{Statistics}
In order to analyze this model, we focus on a few statistics $f, F$, and $\j_0$ deduced from the data above.   These numbers depend only on the lottery itself (e.g., Mega Millions), and not on a particular drawing.

We define $f$ to be the cost of a ticket less the expected winnings from fixed prizes, i.e.,
\begin{equation} \label{f.def}
f := 1 - \sum_{i=1}^c a_i t^\fix_i / t.
\end{equation}
This number is approximately the proportion of lottery sales that go to the jackpot, the pari-mutuel prizes, and ``overhead" (i.e., the cost of lottery operations plus vigorish; around $45\%$ of net sales for the example lotteries in this paper).  It is not quite the same, because we have deducted income taxes from the amounts $a_i$.  Because the $a_i$ are positive, we have $f \le 1$.

We define $F$ to be 
\begin{equation} \label{F.def}
F := f - \sum_{i=1}^d r_i,
\end{equation}
which is approximately the proportion of lottery sales that go to the jackpot and overhead.  Any actual lottery will put some money into one of these, so we have $0 < F \le f \le 1$.

Finally, we put
\begin{equation} \label{j0.def}
\j_0 := Ft.
\end{equation}
We call this quantity the \emph{jackpot cutoff}, for reasons which will become apparent in
Section \ref{j0.sec}.  Table \ref{eg.naive} lists these numbers for our four example lotteries.  We assumed that some ticket can win the jackpot, so $t \ge 1$ and consequently $J_0 > 0$.
\begin{table}[htb]
\begin{center}
\begin{tabular}{|c|rccc|} \hline
Game&\multicolumn{1}{c}{$t$}&$f$&$F$&$\j_0$ \\ \hline
Mega Millions&175,711,536&0.838&0.838&147m \\
Powerball&146,107,962&0.821&0.821&120m \\
Lotto Texas&25,827,165&0.957&0.910&23.5m \\
New Jersey Pick 6&13,983,816&0.947&0.855&11.9m \\ \hline
\end{tabular}
\caption{Some statistics for our example lotteries that hold for all drawings.} \label{eg.naive}
\end{center}
\end{table}





\part{To bet or not to bet: Analyzing the rate of return}
\label{rateofreturn}
\section{Expected rate of return}  \label{E.sec}
\renewcommand{\thetable}{\thesection}

Using the model described in Part \ref{setup}, we now calculate the expected rate of return (eRoR) on a lottery ticket, assuming that a total of $\n$ tickets are sold. The eRoR is
\begin{align}
\eRoR &= - \left( \parbox{.425in}{cost of ticket} \right) + \left( \parbox{1.2in}{expected winnings from fixed prizes} \right) \notag  \\
&\quad + \left( \parbox{1.2in}{expected winnings from pari-mutuel\\ prizes} \right)  + \left( \parbox{1.2in}{expected winnings from the jackpot} \right), \label{naive.E}
\end{align}
where all the terms on the right are measured in units of ``cost of one ticket".
The parameter $f$ defined in \eqref{f.def} is the negative of the first two terms.
  
We will assume that the particular numbers on the other tickets are chosen randomly.  
See \ref{unpopular} below for more on this hypothesis.  With this assumption, the probability 
that your ticket is a jackpot winner together with $w - 1$ of the other tickets is given by the binomial distribution:
\begin{equation}\label{binomial}
\binom{\n - 1}{w-1}\,(1/t)^w\,(1 - 1/t)^{\n - w}.
\end{equation}
In this case the amount you win is $\j / w$.  We therefore define
\[
s(p, N) := \sum_{w \ge 1} \frac1w \binom{\n - 1}{w-1}\,p^w\,(1 - p)^{\n - w},
\]
and now the expected amount won from the jackpot is $\j\,s(1/t, \n)$.  Combining this with a similar computation for the pari-mutuel prizes and with the preceding paragraph, we obtain the formula:
\begin{equation} \label{actual.E0}
\eRoR = -f + \sum_{i=1}^d r_i\, \n \,s(p_i, \n) + \j \, s(1/t, \n),
\end{equation}
where $p_i := t^{\pm}_i / t$ is the probability of winning the $i$th pari-mutuel prize and 
($d$ is the number of pari-mutuel prizes).

The domain of the function $s$ can be extended to allow $\n$ to be real (and not just integral);
one can do this the same way one ordinarily treats binomial coefficients in calculus, or
equivalently by using the closed form expression for $s$ given in Proposition \ref{s1}.

In the rest of this section we use first-year calculus to prove some basic facts about the 
function $s(p,x)$.  An editor pointed out to us that our proofs could be shortened if we
approximated the binomial distribution \eqref{binomial} by a Poisson distribution (with
parameter $\lambda=N/t$).  Indeed this approximation is extremely good for the relevant
values of $N$ and $t$; 
however, since we are able to obtain the results we need from
the exact formula \eqref{binomial}, we will stick with this expression in order
to keep the exposition self-contained.
%

\begin{prop} \label{s1}
For $0<p<1$ and $\n >0$,
\[
s(p, \n) = \frac{1 - (1-p)^\n}{\n}.
\]
\end{prop}

We will prove this in a moment, but first we plug it into \eqref{actual.E0} to obtain:
\begin{equation} \label{actual.E}
\fbox{$\eRoR = -f + \sum_{i=1}^d r_i\,(1 - (1 - p_i)^\n) +  \j \, s(1/t, \n)$}
\end{equation}
Applying this formula to our example drawings from Table \ref{eg.jacks} gives the eRoRs listed in Table \ref{eg.big}. 
\begin{table}[ht]
\begin{tabular}{|r@{/}c@{/}lc|ccc|cc|} \hline
\multicolumn{3}{|c}{Date}&Game&$\j$&$\n$&exp.~RoR&$\j/\j_0$&$\n/\j$ \\ \hline
4&07&2007&Lotto Texas&33.8m&4.2m&$+30\%$&1.44&0.13 \\
3&06&2007&Mega Millions&175m&212m&$-26\%$&1.19&1.22 \\
2&18&2006&Powerball&133m&157m&$-26\%$&1.11&1.18\\
10&19&2005&Powerball&123.3m&161m&$-31\%$&1.03&1.31 \\ \hline
\end{tabular}
\caption{Expected rate of return for some specific drawings, calculated using \eqref{actual.E}.} \label{eg.big}
\end{table}

\begin{proof}[Proof of Prop.~\ref{s1}]
The binomial theorem gives an equality of functions of two variables $x, y$:
\[
(x+y)^n = \sum_{k \ge 0} \binom{n}{k} x^k y^{n-k}.
\]
We integrate both sides with respect to $x$ and obtain
\[
\frac{(x+y)^{n+1}}{n+1} + f(y) = \sum_{k \ge 0} \binom{n}{k} \frac{x^{k+1}}{k+1} y^{n-k}
\]
for some unknown function $f(y)$.  Plugging in $x = 0$ gives:
\[
\frac{y^{n+1}}{n+1} + f(y) = 0\,;
\]
hence for all $x$ and $y$, we have:
\[
\frac{(x+y)^{n+1}}{n+1} - \frac{y^{n+1}}{n+1} = \sum_{k\ge0}\frac1{k+1} \binom{n}{k} x^{k+1} y^{n-k}\,.
\]
The proposition follows by plugging in $x = p$, $y = 1-p$, $\n = n + 1$, and $w = k + 1$.
\end{proof}

We will apply the next lemma repeatedly in what follows.

\begin{lem}\label{power}
If $0<c<1$ then $1-\frac 1c - \ln c <0$.
\end{lem}

\begin{proof}
It is a common calculus exercise to show that $1+z<e^z$ for all $z>0$ (e.g., by using
a power series).  This implies $\ln(1+z)<z$ or, shifting the variable, $\ln z < z-1$ for all $z>1$.
Applying this with $z=1/c$ completes the proof.
\end{proof}

\begin{lem} \label{h.d}
Fix $b \in (0, 1)$.  
\begin{enumerate}
\item 
The function 
\[
h(x,y) := \frac{1 - b^{xy}}x
\]
satisfies $h_x:=\partial h / \partial x < 0$ and $h_y:=\partial h / \partial y > 0$ for $x, y > 0$.  
\item
For every $z>0$, the level set $\{(x, y) \mid h(x,y) = z\}$ intersects the first quadrant
in the graph of a smooth, positive, increasing, concave up function defined on the interval 
$(0,1/z)$.
\end{enumerate}
\end{lem}

\begin{proof}
We first prove (1).  The partial with respect to $y$ is easy: $h_y = -b^{xy} \ln (b) > 0$.
For the partial with respect to $x$, we have:
\[
\frac{\partial}{\partial x} \left[ \frac{1 - b^{xy}}{x} \right] = \frac{b^{xy}}{x^2} \left[ 1-b^{-xy}-\ln(b^{xy}) \right].
\]
Now, $b^{xy}$ and $x^2$ are both positive, and the term in brackets is negative
by Lemma \ref{power} (with $c=b^{xy}$).  Thus $h_x <0$.

To prove (2), we fix $z>0$ and solve the equation $h(x,y)=z$ for $y$, obtaining
\[
y=\frac{\ln(1-xz)}{x\ln b}.
\]
Recall that $0<b<1$, so $\ln b<0$.  Thus when $x>0$ the level curve is the graph 
of a positive smooth function (of $x$) with domain $(0,1/z)$, as claimed.  

Differentiating $y$ with respect to $x$ and rearranging yields
\[
\frac{dy}{dx}=\frac 1{x^2\ln b} \left[ 1-\frac 1 {1-xz} - \ln(1-xz)\right].
\]
For $x$ in the interval $(0,1/z)$, we have $0<1-xz<1$, so Lemma \ref{power} implies
that the expression in brackets is negative.  Thus $\dy/\dx$ is positive;
i.e., $y$ is an increasing function of $x$.

To establish the concavity we differentiate again and rearrange (considerably), 
obtaining
\[
\frac{\ddy}{\dx^2}=\frac{-1}{x^3\ln b}\frac 1{(1-xz)^2}\left[3(1-xz)^2-4(1-xz) + 1-2(1-xz)^2 \ln(1-xz) \right].
\]
As $\ln b<0$ and $x>0$, we aim to show that the expression
\begin{equation}\label{u.expr}
1 - 4c + 3c^2 - 2c^2 \ln(c)
\end{equation}
is positive, where $c :=1-xz$ is between 0 and 1.

To show that \eqref{u.expr} is positive, note first that it is decreasing in $c$.  
(This is because the derivative is $4c(1-\frac 1c - \ln c)$, which is negative by
Lemma~\ref{power}.)  Now, at $c=1$, the value of \eqref{u.expr} is 
0, so for $0<c<1$ we deduce that \eqref{u.expr} is positive.  
Hence $\frac{\ddy}{\dx^2}$ is positive, and the level curve is concave up, as claimed.
\end{proof}

\begin{cor} \label{decr} 
Suppose that $0 < p < 1$.  For $x \ge 0$,
\begin{enumerate}
\item the function $x \mapsto s(p,x)$ decreases from $-\ln(1-p)$ to $0$.
\item the function $x \mapsto x\,s(p, x)$ increases from $0$ to $1$.
\end{enumerate}
\end{cor}

\begin{proof}
(1)  Lemma \ref{h.d} implies that $s(p, x)$ is decreasing, by setting $y = 1$ and $b = 1 - p$.
Its limit as $x\to \infty$ is obvious because $0<1-p<1$, and its limit as $x\to 0$ can
be obtained via l'Hospital's rule.

(2)  Proposition \ref{s1} gives that $x\,s(p, x)$ is equal to $1 - (1 - p)^x$, whence 
it is increasing and has the claimed limits because $0<1-p<1$.  
\end{proof}

\begin{borel}{Unpopular numbers} \label{unpopular}
Throughout this paper, we are assuming that the other lottery players select their tickets randomly.  
This is not strictly true:  in a typical U.S.~lottery drawing, only 70 to 80 percent of the tickets sold are 
``quick picks,'' i.e., tickets whose numbers are picked randomly by computer \cite{PB,MG:US}, whereas 
the others have numbers chosen by the player.  Ample evidence indicates that player-chosen 
numbers are not evenly distributed amongst the possible choices, leading to a strategy:  by playing 
``unpopular" numbers, you won't improve your chances of winning any particular prize, 
but if you do win the jackpot, your chance of sharing it decreases.  This raises your expected return.
See, e.g., \cite{Chernoff}, \cite{HR}, \cite{DrZ}, or the survey in \S4 of \cite{Haigh}.

The interested reader can adjust the results of this paper to account for this strategy.
One way to proceed is to view \eqref{actual.E0} as a lower bound on the eRoR, and then
compute an upper bound by imagining the extreme scenario in which one plays numbers
that are not hand-picked by any other player.  In this case any sharing of the jackpot will 
result from quick picks, which we assume are used by at least 70\% of all players.  Therefore
one can replace $N$ by $N'=0.7N$ in \eqref{binomial} and hence in \eqref{actual.E0}.
The resulting value of eRoR will be an upper bound on the returns available using
the approach of choosing unpopular numbers.
\end{borel}

\section{The jackpot cutoff $\j_0$} \label{j0.sec}

With the results of the preceding section in hand, it is easy to see that the rate of return on ``nearly all" lottery drawings is negative.  How?  We prove an upper bound on the eRoR \eqref{actual.E}.

In \eqref{actual.E}, the terms $1 - (1 - p_i)^\n$ are at most 1.  So we find:
\[
\eRoR \leq -f  + \sum_{i=1}^d r_i+ \j\,s(1/t, \n).
\]
The negative of the first two terms is the number $F$ defined in \eqref{F.def}, i.e.,
\begin{equation} \label{ub.E}
\eRoR \leq -F + \j\,s(1/t, \n).
\end{equation}
But by Corollary \ref{decr}, $s(1/t, \n)$ is less than $s(1/t, 1) = 1/t$, so
\[
\eRoR < -F + \j/t.
\]
In order for the eRoR to be positive, clearly $-F + \j/t$ must be positive, i.e., $\j$ must be greater 
than $Ft=\j_0$.  This is why (in \eqref{j0.def}) we called this number the \emph{jackpot cutoff}.  To summarize:
\begin{equation} \label{jack.cut}
\fbox{\parbox{2.9in}{If $\j < \j_0$, then a lottery ticket is a bad bet, i.e., the expected rate of return  is negative.}}
\end{equation}

With \eqref{jack.cut} in mind, we ignore all drawings with $\j < \j_0$.  This naive and easy-to-check criterion is extremely effective; it shows that ``almost all" lottery drawings have negative eRoR.
The drawings listed in Table \ref{eg.big} are all the drawings in Mega Millions and Powerball (since inception of the current games until the date of writing, December 2008) where the jackpot $\j$ was at least $\j_0$.  There are only 3 examples out of about 700 drawings, so this event is uncommon.  We also include one Lotto Texas drawing with $\j > \j_0$.  (This drawing was preceded by a streak of several such drawings in which no one won the jackpot; the eRoR increased until someone won the April 7, 2007 drawing.)  During 2007, the New Jersey lottery had no drawings with $\j > \j_0$.

Table \ref{eg.big} also includes, for each drawing, an estimate of the number of tickets sold for that drawing.  The precise numbers are not publicized by the lottery operators.  In the case of Powerball, we estimated $\n$ by using the number of tickets that won a prize, as reported in press releases.  For Mega Millions, we used the number of tickets that won either of the two smallest fixed prizes, as announced on the lottery website.  For Lotto Texas, the number of tickets sold is from \cite{TXsales}.

\section{Break-even curves}
\numberwithin{figure}{section}
\renewcommand{\thefigure}{\thesection\textsc{\alph{figure}}}

At this point, for any particular drawing, we are able to apply \eqref{actual.E} to compute the 
expected rate of return.  But this quickly becomes tiresome; our goal in this section is to give 
criteria that---in ``most'' cases---can be used to determine the \emph{sign} of the expected 
rate of return, i.e., whether or not the drawing is a good bet.  These criteria will be easy to
check, making them applicable to large classes of drawings all at once.

Recall that for a given lottery, the numbers $f, r_i, p_i,$ and $t$ (and therefore $F$ and $\j_0$)
are constants depending on the setup of the lottery; what changes from drawing to drawing are
the values of $\j$ and $\n$.  Thus $\j$ and $\n$ are the parameters which control the eRoR for
any particular drawing.  In light of \eqref{jack.cut}, it makes sense to normalize $\j$ by dividing
it by the jackpot cutoff $\j_0$; equation \eqref{jack.cut} says that if $\j/\j_0<1$ then the eRoR is
negative.  It is helpful to normalize $\n$ as well; it turns out that a good way to do this is to 
divide by $\j$, as the quantity $\n/\j$ plays a decisive role in our analysis.  Thus we think of 
ticket sales as being ``large'' or ``small'' only in relation to the size of the jackpot.
We therefore use the variables
\[
x := \n/\j \quad \text{and} \quad y := \j/\j_0
\]
to carry out our analysis, instead of $\n$ and $\j$.  Plugging this in to \eqref{actual.E}---i.e., substituting
$y\j_0$ for $\j$ and $xy\j_0$ for $\n$---gives
\begin{equation} \label{xy.E}
\fbox{$\eRoR = -f + \displaystyle\sum_{i=1}^d r_i (1 - (1-p_i)^{xy\j_0}) + \frac{1 - (1 - p)^{xy\j_0}}x$}.
\end{equation}

For any particular lottery, we can plug in the actual values for the parameters $f$, $r_i$, $p_i$, $p$, and $\j_0$ and plot the level set $\{ \eRoR = 0 \}$ in the first quadrant of the $(x,y)$-plane.  
We call this level set the \emph{break-even curve}, because drawings lying on this curve 
have expected rate of return zero.  This curve is interesting for the gambler because it 
separates the region of the $(x,y)$-plane consisting of drawings with positive rates of 
return from the region of those with negative rates of return.  

\begin{prop} \label{be}
For every lottery as in \S\ref{model}, the break-even curve is the graph of a smooth,
positive function $\ell(x)$ with domain $(0, 1/F)$.  If the lottery has no pari-mutuel 
prizes, then $\ell(x)$ is increasing and concave up.
\end{prop}

\begin{proof}
For a fixed positive $x$, we can use \eqref{xy.E} to view the eRoR as a (smooth)
function of one variable $y$.  At $y = 0$, we have $\eRoR = -f < 0$.  
Looking at \eqref{xy.E}, the non-constant terms are all strictly increasing functions of $y$
(because $x,J_0 > 0$), so eRoR is an increasing function of $y$.  Further,
\[
\lim_{y \mapsto \infty} \eRoR = -f + \sum_{i=1}^d r_i + \frac1x = -F + \frac1x.
\]
Therefore, by the intermediate value theorem, for each $x \in (0, 1/F)$ there is a
$y>0$ such that $\eRoR = 0$; this value is unique by the mean value theorem.  Moreover
for $x \ge 1/F$ the eRoR is always negative.  This proves that the break-even curve
is the graph of a positive function $\ell(x)$ with domain $(0, 1/F)$.

Now, observe that for fixed $y>0$, it is in general unclear whether eRoR is increasing 
in $x$.  One expects it to decrease, because $x$ is proportional to ticket sales.
However, in \eqref{xy.E}, the terms in the sum (which are typically small) are increasing 
in $x$ whereas only the last term is decreasing in $x$ (by Lemma \ref{h.d}).
If there are no pari-mutuel prizes, then the sum disappears, and the eRoR is
decreasing in $x$.  Specifically, in this case the graph of $\ell(x)$ is the level curve 
of the function $h(x,y) = (1 - b^{xy})/x$ at the value $f$, where $b = (1-p)^{\j_0}$.  
By Lemma \ref{h.d}, $\ell(x)$ is smooth, increasing, and concave up.

Nevertheless, even when there are pari-mutuel prizes, the function $\ell(x)$ is still
smooth.  This is because, as we showed in the first paragraph of this proof, for any 
fixed $x>0$, eRoR is increasing in $y$.  This means the $y$-derivative of eRoR is 
never zero (for $0<x<1/F$), so we can invoke the implicit function theorem to 
conclude that its level sets are graphs of smooth functions of $x$ on this interval.  
The break-even curve is one such level set.
\end{proof}

The break-even curves for Mega Millions,
Powerball, Lotto Texas, and New Jersey Pick 6 are the lighter curves in Figure~\ref{be.fig}.
(Note that they are so similar that it is hard to tell that there are four of them!)

The two bold curves in Figure \ref{be.fig} are the level curves
\begin{equation} \label{lb}
U:=\left\{(x,y) \, \left|\   -1 + \frac{1 - 0.45^{xy}}{x} = 0\right. \right\}
\end{equation}
and
\begin{equation} \label{ub}
L:=\left\{(x,y) \,\left|\  -0.8 + \frac{1 - 0.36^{xy}}{x} = 0\right.\right\}.
\end{equation}
Their significance is spelled out in the following theorem, which is our core result.
We consider a lottery to be \emph{major} if there are at least 500 distinct tickets;
one can see from the proofs that this particular bound is almost arbitrary.

\begin{figure}[ht]
\[
\includegraphics[width=3in]{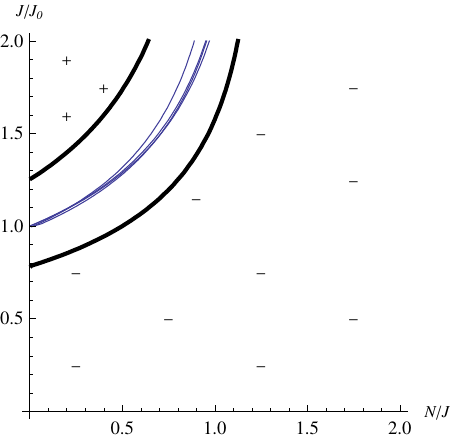}
\]
\caption{Illustration of Theorem \ref{crux}.  The break-even curve for any lottery must lie in the region between the two bold curves.  The lighter curves between the two bold curves are the break-even curves for Mega Millions, Powerball, Lotto Texas, and New Jersey Pick 6.  Drawings in the region marked $-$ have a negative eRoR and those in the region marked $+$ have a positive eRoR.}\label{be.fig}
\end{figure}

\begin{thm} \label{crux}
For any major lottery with $F \ge 0.8$ (i.e., that roughly speaking pays out on average less than $20\%$ of its revenue in prizes other than the jackpot) we have:
\begin{enumerate}
\renewcommand{\theenumi}{\alph{enumi}}
\item The break-even curve lies in the region bounded by the curves $U$ and $L$ and the $y$-axis. \label{crux.a}
\item \label{crux.b} Any drawing with $(x,y)$ above the break-even curve has positive eRoR. 
\item \label{crux.c} Any drawing with $(x,y)$ below the break-even curve has negative eRoR. 
\end{enumerate}
In particular, for any drawing of any such lottery, if $(x,y)$ is above $U$ then the eRoR is
positive, and if $(x,y)$ is below or to the right of $L$ then the eRoR is negative.
\end{thm}

The hypotheses of the theorem are fairly weak.  Being major means that the lottery has at least 500 possible distinct tickets; our example lotteries all have well over a million distinct tickets.  All of our example lotteries have $F \ge 0.82$.

We need the following lemma from calculus.

\begin{lem}\label{1/e}
The function $g(t)=(1-1/t)^t$ is increasing for $t>1$, and the limit is $1/e$.
\end{lem}

\begin{proof}
The limit is well known.  We prove that $g$ is increasing, which is an exercise in 
first-year calculus.\footnote{This function is one of those for which computers give a misleading plot.  For large values of $t$, say $t > 10^8$, Mathematica shows a function that appears to be oscillating.}
In order to take the derivative we first take the logarithm:
\[
\ln g(t) = t\ln(1-1/t)
\]
Now differentiating both sides yields
\begin{equation} \label{ln}
\frac {g'(t)}{g(t)} = \ln(1-1/t) + t \frac {1/t^2}{1-1/t} = \ln(1-1/t) + \frac 1 {t-1}.
\end{equation}
We want to show that $g'(t)>0$.  Since the denominator on the left, $g(t)$,
is positive, it suffices to show that the right side is positive.  To do this we 
rearrange the logarithm as $\ln(1-1/t)=\ln(\frac{t-1}{t})=\ln(t-1)-\ln t$. 
Now, using the fact that 
\[
\ln t = \int _1^t \frac 1 x \,\mathrm{d}x,
\]
the right side of \eqref{ln} becomes 
\[
-\int _{t-1}^t \frac 1 x \,\dx + \frac 1 {t-1},
\]
which is evidently positive, because the integrand $1/x$ is less than 
$1/(t-1)$ on the interval of integration.
\end{proof}

We can now prove the theorem.

\begin{proof}[Proof of Theorem \ref{crux}]
Parts \eqref{crux.b} and \eqref{crux.c} follow from the proof of Proposition \ref{be}: 
for a fixed $x$, the eRoR is an increasing function of $y$.

To prove (a), we first claim that for any $c > 0$,
\begin{equation} \label{crux.1}
1 - 0.45^c < c \j_0 \, s(1/t, c\j_0) < 1 - 0.36^c.
\end{equation}
Recall that $\j_0=Ft$.  We write out $s$ using Proposition \ref{s1}:
\[
c\j_0\, s(1/t, c \j_0) = 1 - (1- 1/t)^{cFt}.
\]
By Lemma \ref{1/e}, the function
$t\mapsto (1-1/t)^t$ is increasing.  Thus, since $t\geq 500$ by assumption, we have 
\[
(1-1/t)^t\ge 0.998^{500}>0.36.
\]
For the lower bound in \eqref{crux.1}, Lemma \ref{1/e} implies that $(1 - 1/t)^t$ is at most $1/e$.  Putting these together, we find:
\[
1 - e^{-cF} < c \j_0\, s(1/t, c \j_0) < 1 - 0.36^{cF}.
\]
Finally, recall that $F\geq0.8$;
plugging this in (and being careful with the inequalities) now establishes \eqref{crux.1}.

Looking at equation \eqref{actual.E} for the eRoR, we observe that the pari-mutuel rates $r_i$ are nonnegative, so for a lower bound we can take $p_i = 0$ for all $i$.  Combined with the upper bound \eqref{ub.E}, we have:
\[
-f + J\,s(1/t, N) \le \eRoR \le -F + J \, s(1/t, N).
\]
Replacing $N$ with $xy J_0$ and $J$ with $xyJ_0 / x$ and applying \eqref{crux.1} with $c = xy$ gives:
\begin{equation} \label{crux.2}
-1 + \frac{1 - 0.45^{xy}}{x} < \eRoR < -0.8 + \frac{1 - 0.36^{xy}}{x}.
\end{equation}

By Lemma \ref{h.d}, the partial derivatives of the upper and lower bounds in \eqref{crux.2} 
are negative with respect to $x$ and positive with respect to $y$.  This implies \eqref{crux.a} 
as well as the last sentence of the theorem.
\end{proof}

\begin{figure}[h]
\[
\includegraphics[width=3in]{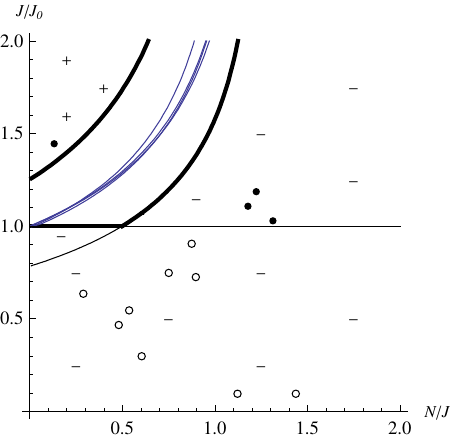}
\]
\caption{Refinement of Figure \ref{be.fig}.  Drawings in the regions marked with $+$'s have positive eRoR and those in regions marked with $-$'s have negative eRoR.}\label{summ}
\end{figure}

We can enlarge the negative region in Figure \ref{be.fig} somewhat by incorporating \eqref{jack.cut}, which says that the eRoR is negative for any drawing with $y = \j/\j_0 < 1$.  
In fact \eqref{jack.cut} implies that the break-even curve for a lottery will not intersect the line 
$y=1$ except possibly when no other tickets are sold, i.e., when $\n=1$ or equivalently 
$x=1/\j\approx 0$.  The result is Figure \ref{summ}.  

Figure \ref{summ} includes several data points for actual drawings which have occurred.  
The four solid dots represent the four drawings from Table \ref{eg.big}.  The circles in the bottom 
half of the figure are a few typical Powerball and Mega Millions drawings.  In all the drawings that we examined, the only ones we found in the inconclusive region (between the bold curves)
were some of those leading up to the positive Lotto Texas drawing plotted in the figure.

\renewcommand{\thefigure}{\thesection}

\section{Examples} \label{applications}

We now give two concrete illustrations of Theorem \ref{crux}.  Let's start with the good news.

\begin{borel}{Small ticket sales} \label{small.eg}
The point $(0.2,1.4)$ (approximately) is on the curve $U$ defined in \eqref{lb}.
Any drawing of any lottery satisfying 
\begin{enumerate}
\item $\n < 0.2\,\j$ and \label{small.nb}
\item $\j >  1.4 \, \j_0$ \label{small.jb}
\end{enumerate}
is above $U$ and so will have positive eRoR by Theorem \ref{crux}.
(This points satisfying (1) and (2) make up the small shaded rectangle on the left side of Figure~\ref{rectangles.fig}.)  We chose
to look at this point because many state lotteries, such as Lotto Texas, tend to satisfy
\eqref{small.nb} every week.  So to find a positive eRoR, this is the place to look:
just wait until $\j$ reaches the threshold \eqref{small.jb}.  
\end{borel}

\begin{figure}[ht]
\[
\includegraphics[height=3in]{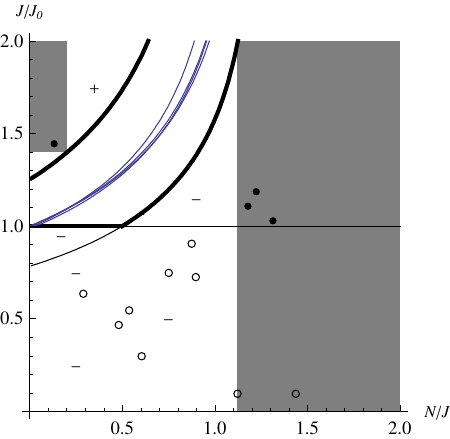}
\]
\caption{Figure \ref{summ} with the regions covered by the Small Ticket Sales and the Large 
Ticket Sales examples shaded.}\label{rectangles.fig}
\end{figure}

\begin{borel}{Large ticket sales} \label{big.eg}
On the other hand, the point $(1.12,2)$ is on the curve $L$ defined in \eqref{ub}; therefore any drawing of any lottery with 
\begin{enumerate}
\item $\n >1.12\,\j$ and \label{large.nb}
\item $\j < 2\,\j_0$ \label{large.jb}
\end{enumerate}
lies below or to the right of $L$ and will have a negative eRoR.  These drawings make up the large 
rectangular region on the right side of Figure~\ref{rectangles.fig}.
Here we have chosen to focus on this particular rectangle for two reasons.
Mega Millions and Powerball tend to have large ticket sales (relative to $\j$); specifically,
$\n/\j$ has exceeded $1.12$ every time $\j$ has exceeded $\j_0$.  Moreover, no
drawing of any lottery we are aware of has ever come close to violating \eqref{large.jb}.
In fact, for lotteries with large ticket sales (such as Mega Millions and Powerball), the largest 
value of $\j/\j_0$ we have observed is about $1.19$, in the case of the Mega Millions drawing in 
Table \ref{eg.big}.  Thus no past drawing of Mega Millions or Powerball has 
ever offered a positive eRoR.

Of course, if we are interested in Mega Millions and Powerball specifically, then we
may obtain stronger results by using their actual break-even curves, rather than the
bound $L$.  We will do this in the next section to argue that in all likelihood these 
lotteries will never offer a good bet.
\end{borel}

In each of these examples, our choice of region (encoded in the hypotheses 
(1) and (2)) is somewhat arbitrary.  The reader who prefers a different rectangular region 
can easily cook one up:  just choose a point on $U$ or $L$ to be the (lower right or upper left)
corner of the rectangle and apply Theorem~\ref{crux}.  At the cost of slightly more complicated 
(but still linear) hypotheses, one could prove something about various triangular regions 
as well.

\section{Mega Millions and Powerball} \label{MMPB.sec}

Mega Millions and Powerball fall under the Large Ticket Sales example (\ref{big.eg})
of the previous section, and indeed they have never offered a positive eRoR.  We now 
argue that \emph{in all likelihood, no future drawing of either of these lotteries will ever 
offer a positive expected rate of return.}

\begin{borel}{Mega Millions / Powerball}\label{MMPB.eg}
Note that for these lotteries we know the exact break-even curves, so we needn't
use the general bound $L$.  Using the data from Table~\ref{eg.naive}, we find that
the break-even curve for Mega Millions is given by
\[
Z_{MM} := \left\{-0.838 + \frac { 1- 0.43^{xy} } {x} =0 \right\}
\]
and the break-even curve for Powerball is given by 
\[
Z_{PB} := \left\{-0.821 + \frac { 1- 0.44^{xy}} x =0 \right\}.
\]

The point $(1,2)$ is below both of these curves, as can be seen by plugging 
in $x=1$ and $y=2$.  Thus, as in Theorem \ref{crux}, \emph{any Mega Millions or Powerball
drawing has negative eRoR if
\begin{enumerate}
\item $N > J$ and
\item $J < 2\,J_0$.
\end{enumerate}}
\end{borel}

\begin{borel}{Why Mega Millions and Powerball will always be bad bets}\label{big.bad}
As we have mentioned already, every time a Mega Millions or Powerball
jackpot reaches $\j_0$, the ticket sales $\n$ have easily exceeded the jackpot $\j$.
Assuming this trend will continue, the preceding example shows that a gambler seeking 
a positive rate of return on a Mega Millions
or Powerball drawing need only consider those drawings where the jackpot $\j$ is at least $2\j_0$.
(Even then, of course, the drawing is not guaranteed to offer a positive rate of return.)  We
give a heuristic calculation to show that the jackpot will probably never be so large.

Since the
inception of these two games, the value of $\j/\j_0$ has exceeded 1 only three times, in the last three drawings in Table \ref{eg.big}.
The maximum value attained so far is 1.19.
What would it take for $\j/\j_0$ to reach 2?

We will compute two things:  first, the probability that a large jackpot ($\j\geq\j_0$) rolls over,
and second, the number of times this has to happen for the jackpot to reach $2\j_0$.
The chance of a rollover is the chance that the jackpot is not won, i.e.,
\[
(1-1/t)^\n
\]
where $\n$ is the number of tickets sold.  This is a decreasing function of $\n$, so in order to 
find an upper bound for this chance, we need a lower bound on $\n$ as a function of
$\j$.  For this we use the assumption \eqref{large.nb} of Section \ref{big.eg}, which historically has
been satisfied for all large jackpots.  Since a rollover can only cause the jackpot to increase, 
the same lower bound on $\n$ will hold for all future drawings until the jackpot is won.
Thus if the current jackpot is $\j\geq \j_0$, the chance of rolling over $k$ times is at most
\[
(1-1/t)^{k\,\j_0}.
\]
Now, $\j_0=Ft$ and since $t$ is quite large, we may approximate $(1-1/t)^t$ by its limiting 
value as $t$ tends to infinity, which is $1/e$.  Thus the probability above is roughly $e^{-kF}$.
Since $F\geq 0.82$, we conclude that once the jackpot reaches $\j_0$,
the chance of it rolling over $k$ times is no more than $e^{-.82k}$.

Let us now compute the number $k$ of rollovers it will take for the jackpot to reach $2\j_0$.
Each time the jackpot ($\j$) rolls over (due to not being won), the jackpot for the next drawing 
increases, say to $\j'$.  The amount of increase depends on the ticket sales, as this is the only source 
of revenue and a certain fraction of
revenue is mandated to go toward the jackpot.  So in order to predict how much the jackpot
will increase, we need a model of ticket sales as a function of the jackpot.  It is not clear how
best to model this, since the data are so sparse in this range.  For smaller jackpots, evidence
\cite{deboer} indicates that ticket sales grow as a quadratic function of the jackpot, so one possibility
is to extrapolate to this range.  On the other hand, no model can work for arbitrarily large jackpots,
because of course at some point the sales market will be saturated.  For the purposes of this
argument, then, we note that the ratio $\j'/\j$ has never been larger than about $1.27$ for
any reasonably large jackpot, and we use this as an upper bound on the rolled-over jackpot.
This implicitly assumes a linear upper bound on ticket sales as a function of the jackpot,
but the bound is nevertheless generous by historical standards.\footnote{It was pointed out
to us by Victor Matheson that over time lottery ticket sales are trending downward \cite{deboer}, so 
our upper bound on all past events is likely to remain valid in the future.}

So, suppose that each time a large jackpot $\j$ is rolled over, the new jackpot $\j'$ is less 
than $1.27\j$.  Then, even the jackpot corresponding to the largest value of $\j/\j_0$ (namely
$\j/\j_0=1.19$) would have had to roll over three more times before $\j$ would have surpassed
$2\j_0$.  Thus we should evaluate $e^{-0.82k}$ with $k=3$.

Plugging in $k=3$ shows that once the jackpot reaches $\j_0$, the probability that it
reaches $2\j_0$ is at most $e^{-0.82\cdot 3}<1/11$.
So if the jackpot of one of these lotteries reaches $\j_0$ about once every two years (a reasonable estimate by historical standards), then one would expect it to reach $2\j_0$ about once every 22 years, which is longer than the life 
span of most lotteries.  This is the basis of our claim that Mega Millions and Powerball
are unlikely to ever offer a positive eRoR.

As a final remark, we point out that even in the ``once every 22 years" case where the jackpot 
exceeds $2\j_0$, one only concludes that the hypotheses of Example \ref{MMPB.eg} are not satisfied; one still needs to check whether the expected rate of return is positive. 
And in any case, there still remains the question of whether buying lottery tickets is a good investment.
\end{borel}

\part{To invest or not to invest:  Analyzing the risk}
\label{risk}

In the last few sections we have seen that good bets in the lottery
do exist, though they may be rare.  We also have an idea of how
to find them.  So if you are on the lookout and you spot a good
bet, is it time to buy tickets?  In the rest of the paper we address
the question of whether such an opportunity would actually make
a good investment.

\section{Is positive rate of return enough?} \label{port}
\setcounter{table}{0}

The Lotto Texas drawing of April 7, 2007 had a huge rate of return of 30\% over a very 
short time period (which for sake of argument we call a week).  Should you buy tickets for such 
a drawing?  The naive investor would see the high rate of return and immediately say 
``yes."\footnote{The naive view is that rate of return is all that matters.  This is not just a straw man; 
the writers of mutual fund prospectuses often capitalize on this naive view by reporting historical rates of return, but not, say, correlation with the S\&P 500.}
But there are many 
risky investments with high expected rates of return that investors typically choose not to invest 
in directly, like oil exploration or building a time machine.  How should one make these decisions?

This is an extreme example of a familiar problem:  you want to invest in a combination of various assets with different rates of return and different variances in those rates of return.  Standard examples of such assets are a certificate of deposit with a low (positive) rate of return and zero variance, bonds with a medium rate of return and small variance, or stocks with a high rate of return and large variance.  How can you decide how much to invest in each?  Undergraduate economics courses teach a method, known as \emph{mean-variance portfolio analysis}, by which the investor can choose the optimal proportion to invest in each of these assets.  (Harry Markowitz and William Sharpe shared the Nobel Prize in economics in 1990 for their seminal work in this area.)  We will apply this same method to compare lottery tickets with more traditional risky assets.  This will provide both concrete investment advice and a good illustration of how to apply the method, which requires little more than basic linear algebra.

The method is based on certain assumptions about your investment preferences.  You want to pick a  \emph{portfolio}, i.e., a weighted combination of the risky assets where the weights sum to 1.  We assume:
\begin{enumerate}
\item When you decide whether or not to invest in an asset, the only attributes you consider are its expected rate of return, the variance in that rate, and the covariance of that rate with the rates of return of other assets.\footnote{This assumption is appealing from the point of view of mathematical modeling, but there are various objections to it, hence also to mean-variance analysis.  But we ignore these concerns  because mean-variance analysis is ``the de-facto standard in the finance profession" \cite[\S2.1]{Brandt}.}
\item You prefer portfolios with high rate of return and low variance over portfolios with low rate of return and high variance.  (This is obvious.)
\item Given a choice between two portfolios with the same variance, you prefer the one with the higher rate of return.
\item \label{subtle} Given a choice between two portfolios with the same rate of return, you prefer the one with the lower variance.
\end{enumerate}
From these assumptions, we can give a concrete decision procedure for picking an optimal portfolio.

But before we do that, we first justify (4).
In giving talks about this paper, we have found that \eqref{subtle} comes as a surprise to some mathematicians; they do not believe that it holds for every rational investor.  Indeed, if an investor acts to maximize her expected amount of money, then two portfolios with the same rate of return would be equally desirable.  But economists assume that rational investors act to maximize their expected \emph{utility}.\footnote{Gabriel Cramer of Cramer's rule fame once underlined this distinction by writing: ``mathematicians evaluate money in proportion to its quantity while ... people with common sense evaluate money in proportion to the utility they can obtain from it." \cite[p.~33]{Bernoulli}.}

In mathematical terms, an investor  derives utility (pleasure) $U(x)$ from having $x$ dollars, and the investor seeks to maximize the expected value of $U(x)$.  Economists assume that $U'(x)$ is positive, i.e., the investor always wants more money.  They call this axiom \emph{non-satiety}.  They also assume that $U''(x)$ is negative, meaning that someone who is penniless values a hundred dollars more than a billionaire does.  These assumptions imply \eqref{subtle}, as can be seen by Jensen's inequality:  given two portfolios with the same rate of return, the one with lower variance will have higher expected utility.   As an alternative to using Jensen's inequality, one can examine the Taylor polynomial of degree 2 for $U(x)$.  The great economist Alfred Marshall took this second approach in Note IX of the Mathematical Appendix to \cite{Marshall}.\footnote{That note ends with: ``... experience shows that [the pleasures of gambling] are likely to engender a restless, feverish character, unsuited for steady work as well as for the higher and more solid pleasures of life."}  Economists summarize property \eqref{subtle} by saying that \emph{rational investors are risk-averse}.  


\section{Example of portfolio analysis}

We consider a portfolio consisting of typical risky assets: the iShares Barclays aggregate bond fund (AGG), the MSCI EAFE index (which indexes stocks from outside of North America), the FTSE/NAREIT all REIT index (indexing U.S.\  Real Estate Investment Trusts), the Standard \& Poor 500 (S\&P500), and the NASDAQ Composite index.  We collected weekly adjusted returns on these investments for the period January 31, 1972 through June 4, 2007---except for AGG, for which the data began on September 29, 2003.  This time period does not include  the current recession that began in December 2007.  The average weekly rates of return and the covariances of these rates of return are given in Table \ref{cov}.
\begin{table}[htb]
\begin{tabular}{|cc|c|rrrrr|} \hline
\#&Asset&eRoR&AGG&EAFE&REIT&S\&P500&NASDAQ \\ \hline
1&AGG&.057&.198&.128&.037&$-.077$&$-.141$ \\
2&EAFE&.242&&5.623&2.028&.488&.653 \\
3&REIT&.266&&&4.748&.335&.419 \\
4&S\&P500&.109&&&&10.042&10.210 \\
5&NASDAQ&.147&&&&&12.857 \\ \hline
\end{tabular}
\caption{Typical risky investments.  The third column gives the expected weekly rate of return in \%.  Columns 4--8 give the covariances (in $\%^2$) between the weekly rates of return.  The omitted covariances can be filled in by symmetry.} \label{cov}
\end{table}

We choose one of the simplest possible versions of mean-variance analysis.  We assume that you can invest in a risk-free asset---e.g., a short-term government bond or a savings account---with positive rate of return $R_F$.  The famous ``separation theorem" says that you will invest in some combination of the risk-free asset and an efficient portfolio of risky assets, where the risky portfolio depends only on $R_F$ \emph{and not on your particular utility function.}
See \cite{Sharpe} or any book on portfolio theory for
more on this theorem.

We suppose that you will invest an amount $i$ in the risky portfolio, with units chosen so that $i = 1$ means you will invest the price of 1 lottery ticket.  We describe the risky portfolio with a vector $X$ such that your portfolio contains $i \, X_k$ units of asset $k$.  If $X_k$ is negative, this means that you ``short" $i\, |X_k|$ units of asset $k$.  We require further that $\sum_k |X_k| = 1$, i.e., in order to sell an asset short, you must put up an equal amount of cash as collateral.  (An economist would say that we allow only ``Lintnerian" short sales, cf.~\cite{Lintner}.)

Determining $X$ is now an undergraduate exercise.  We follow \S{II} of \cite{Lintner}.  Write $\mu$ for the vector of expected returns, so that $\mu_k$ is the eRoR on asset $k$, and write $C$ for the symmetric matrix of covariances in rates of return; initially we only consider the 5 typical risky assets from Table \ref{cov}. Then by \cite[p.~21, (14)]{Lintner},
\begin{equation} \label{Xdef}
X = \frac{Z}{\sum_k |Z_k | } \quad \text{where} \quad 
Z := C^{-1} \left( \mu - R_F \left( \begin{smallmatrix} 1 \\ 1 \\ \vdots\\ 1 \end{smallmatrix} \right) \right).
\end{equation}
For a portfolio consisting of the 5 typical securities in Table \ref{cov}, we find, with numbers rounded to three decimal places:
\begin{equation}\label{Zvec}
Z = \left( \begin{smallmatrix}
0.277- 5.118 R_F \\
0.023+ 0.013 R_F \\
0.037-  0.165 R_F \\ 
-0.009 - 0.014 R_F \\
0.019- 0.118 R_F
\end{smallmatrix} \right).
\end{equation}

\section{Should you invest in the lottery?} \label{lottery.port}

We repeat the computation from the previous section, but we now include a particular lottery drawing as asset 6.  
We suppose that the drawing has eRoR $R_L$ and variance $v$.  Recall that the eRoR is given by \eqref{actual.E}.  One definition of the variance is $E(X^2) - E(X)^2$; here is a way to quickly estimate it.

\begin{borel}{Estimating the variance} \label{var}
For lottery drawings,  $E(X^2)$ dwarfs $E(X)^2$ and by far the largest contribution to the former term comes from the jackpot.  Recall that jackpot winnings might be shared, so by taking into account the odds for the various possible payouts, we get the following estimate for the variance $v_1$ on the rate of return for a single ticket, where $w$ denotes the total number of tickets winning the jackpot, including yours:
\[
v_1 \approx \sum_{w \ge 1} 100^2 \left( \frac{\j}{w} - 1\right)^2 \binom{N-1}{w-1} p^w (1-p)^{N-w}.
\]
(The factor of $100^2$ converts the units on the variance to $\%^2$, to match up with the notation in the preceding section.)  Although this looks complicated, only the first few terms matter.

There is an extra complication.  The variance of an investment in the lottery depends on how many tickets you buy, and we avoid this worry by supposing that you are buying shares in a syndicate that expects to purchase a fixed number $S$ of lottery tickets.  In this way, the variance $v$ of your investment in the lottery is the same as the variance in the syndicate's investment.  Assuming the number $S$ of tickets purchased is small relative to the total number of possible tickets, the variance $v$ of your investment is approximately $v_1 / S$.
\end{borel}

\newcommand{\cut}{\theta}

\medskip
With estimates of the eRoR $R_L$ and the variance $v$ in hand, we proceed with the portfolio analysis.
We write $\hat{\mu}$ for the vector of expected rates of return and $\hat{C}$ for the matrix of covariances, so that
\[
\hat{\mu} = \begin{pmatrix} \mu \\ R_L \end{pmatrix} \quad \text{and} \quad \hat{C} = \begin{pmatrix} C & 0 \\ 0 & v \end{pmatrix}.
\]
We have:
\[
\hat{Z} = \hat{C}^{-1} \left(\hat{\mu} - R_F  \left( \begin{smallmatrix} 1 \\ 1 \\ \vdots\\ 1 \end{smallmatrix} \right) \right) = 
\begin{pmatrix}
Z \\
(R_L - R_F)/v
\end{pmatrix}.
\]
We may as well assume that $R_L$ is greater than $R_F$; otherwise, buying lottery tickets increases your risk and gives you a worse return than the risk-free asset.  Then the last coordinate of $\hat{Z}$ is positive.  This is typical in that an efficient portfolio contains some amount of ``nearly all" of the possible securities.  In the real world, finance professionals do not invest in assets where mean-variance analysis suggests an investment of less than some fraction $\cut$ of the total.  Let us do this.  With that in mind, should you invest in the lottery?

\begin{negthm}
Under the hypotheses of the preceding paragraph, if
\[
v \ge \frac{R_L - R_F}{0.022 \cut} \, ,
\]
then an efficient portfolio contains a negligible fraction of lottery tickets.
\end{negthm}

Before we prove this Negative Theorem, we apply it to the Lotto Texas drawing from Table \ref{eg.big}, so, e.g., $R_L$ is about 30\% and the variance for a single ticket $v_1$ is about $4 \times 10^{11}$ (see Section \ref{var}).  Suppose that you have \$1000 to invest for a week in some combination of the 5 typical risky assets or in tickets for such a lottery drawing.  We will round our 
investments to whole numbers of dollars, so a single investment of less than $50$ cents will be
considered negligible.  Thus we take $\cut = 1/2000$.  We have:
\[
\frac{R_L - R_F}{.022 \cut} < \frac{30}{1.1 \times 10^{-5}} < 2.73 \times 10^6.
\]

In order to see if the Negative Theorem applies, we want to know whether $v$ is bigger than $2.73$ million.  As in Section \ref{var}, we suppose that you are buying shares in a syndicate that will purchase $S$ tickets, so the variance $v$ is approximately $(4 \times 10^{11})/S$.  That is, it appears that the syndicate would have to buy around
\[
S \approx \frac{4 \times 10^{11}}{2.73 \times 10^6} \approx 145,\!000 \text{\ tickets}
\]
to make the (rather coarse) bound in the Negative Theorem fail to hold.  

In other words, according to the Negative Theorem, if the syndicate buys fewer than $145,\!000$ tickets, our example portfolio would contain zero lottery tickets.  If the syndicate buys more than that number of tickets, then it might be worth your while to get involved.

\begin{proof}[Proof of the Negative Theorem]
We want to prove that $| \hat{X}_6 | < \cut$.  We have:
\[
| \hat{X}_6 | = \frac{ |\hat{Z}_6|}{\sum_k |\hat{Z}_k|}
 < \frac{|\hat{Z}_6|}{|\hat{Z}_2|}.
\]
(The inequality is strict because $\hat{Z}_6$ is not zero.)  Since $R_F$ is positive, \eqref{Zvec} gives $\hat{Z}_2 > 0.022$.  Plugging in the formula $| \hat{Z}_6 | = (R_L - R_F)/v$, we find:
\[
| \hat{X}_6 | < \frac{R_L - R_F}{0.022v}.
\]
By hypothesis, the fraction on the right is at most $\cut$.
\end{proof}

\section{Conclusions}

So, how can you make money at the lottery?

\subsection*{Positive return} If you just want a positive expected rate of return, then our results in Section \ref{MMPB.sec} say to avoid Mega Millions and Powerball.  Instead, you should buy tickets in lotteries where the ticket sales are a small fraction of the jackpot.  And furthermore, you should only do so when the jackpot is unusually large---certainly the jackpot has to be bigger than the jackpot cutoff $\j_0$ defined in \eqref{j0.def}.  

To underline these results, we point out that the Small Ticket Sales example \ref{small.eg} gives concrete criteria that guarantee that a lottery drawing has a positive rate of return.  We remark that even though these drawings have high variances,  a positive rate of return is much more desirable than, say,
casino games like roulette, keno, and slot machines, which despite their consistent negative 
rates of return are nonetheless extremely popular.

\subsection*{Positive return and buyouts} For drawings with a positive rate of return, one can obtain a still higher rate of return by ``buying out" the lottery, i.e., buying one of each of the possible distinct tickets.  We omit a detailed analysis here because of various complicating factors such as the tax advantages, increase in jackpot size due to extra ticket sales, and a more complicated computation of expected rate of return due to the lower-tier pari-mutuel prizes.  The simple computations we have done do not scale up to include the case of a buyout.  See \cite{GM:fair} for an exhaustive empirical study.

In any case, buying out the lottery requires a large investment---in the case of Lotto Texas, about \$26 million.  This has been attempted on several occasions, notably with Australia's New South Wales lottery in 1986, the Virginia (USA) Lottery in 1992 \cite{Mandel95}, and the Irish National Lottery in 1992 \cite{pole}.\footnote{Interestingly, in both the Virginia and Irish drawings, buyout organizers won the jackpot despite only buying about 70\% \cite{Mandel92} and 92\% \cite{pole:apt} of the possible tickets, respectively.}  Buying out the lottery may be tax advantaged because the full cost of buying the tickets may be deductible from the winnings; please consult a tax professional for advice.  On the other hand, buying out the lottery incurs substantial operational overhead in organizing the purchase of so many tickets, which typically must be purchased by physically going to a lottery retailer and filling out a play slip.

We also point out that in some of the known buyout attempts
the lottery companies resisted paying out the prizes, claiming that the practice of buying out
the game is counter to the spirit of the game (see \cite{Mandel95}, \cite{pole}).  In both the cited cases there was a settlement, but one should be aware
that a buyout strategy may come up against substantial legal costs.

\subsection*{Investing} If you are seeking consistently good investment opportunities, then our results in Section \ref{lottery.port} suggest that this doesn't happen in the lottery, due to the astounding variances in rates of return on the tickets.  

What if a syndicate intends to buy out a lottery drawing, and there is a positive expected rate of return?  Our mean-variance analysis suggests that you should invest a small amount of money in such a syndicate.

\subsection*{Alternative strategies}
Buying tickets is not the only way to try to make money off the lottery.  The corporations that operate Mega Millions and Powerball are both profitable, and that's even after giving a large fraction of their gross income (35\% in the case of Mega Millions \cite{MM}) as skim to the participating states. Historically speaking, lotteries used to offer a better rate of return to players \cite{sprowls}, possibly because the lottery only had to provide profits to the operators.  With this in mind, today's entrepreneur could simply run his or her own lottery with the same profit margins but returning to the players (in the form of better prizes) the money that the state would normally take.  This would net a profit for the operator and give the lottery players a better game to play. Unfortunately, we guess that running one's own lottery is illegal in most cases.  But a similar option may be open to casinos: a small modification of keno to allow rolling jackpots could combine the convenience and familiarity of keno with the excitement and advertising power of large jackpots. 

\bigskip

\noindent{\small\textbf{Acknowledgments.} We thank Amit Goyal and Benji Shemmer for sharing
their helpful insights; Bill Ziemba and Victor Matheson for valuable feedback on an earlier draft;
and Rob O'Reilly for help with finding the data in Table \ref{cov}.  We owe Ron Gould a beer for instigating this project.  The second author was partially supported by NSF grant DMS-0653502.}

\providecommand{\bysame}{\leavevmode\hbox to3em{\hrulefill}\thinspace}
\providecommand{\MR}{\relax\ifhmode\unskip\space\fi MR }
\providecommand{\MRhref}[2]{%
  \href{http://www.ams.org/mathscinet-getitem?mr=#1}{#2}
}
\providecommand{\href}[2]{#2}

\end{document}